\documentclass[a4paper]{article}
\usepackage{}
\usepackage{pifont}
\usepackage{amsfonts}

\usepackage{amsmath}
\usepackage{mathrsfs}
\usepackage{amssymb}
\usepackage{amsthm}

\numberwithin{equation}{section}

\newtheorem{Theorem}{Theorem}[section]
\newtheorem{Lemma}{Lemma}[section]
\newtheorem{Remark}{Remark}[section]

\newtheorem{Definition}{Definition}[section]
\newtheorem{Proposition}{Proposition}[section]

\begin{document}

\title{Energy Scattering for Schr\"{o}dinger Equation with Exponential Nonlinearity in Two Dimensions\thanks{This work is supported by China Scholarship Council} }

\author{Shuxia Wang\thanks{{\sl E-mail address: }wangshuxia@pku.edu.cn}\\
 \it\small School of Mathematical Sciences, Peking University,\\
 \it \small Beijing 100871,  People's Republic of China}
\date{}
\maketitle

\begin{abstract}
When the spatial dimensions $n$=2,  the  initial data $u_0\in H^1$ and the Hamiltonian $H(u_0)\leq 1$,
we prove that
the scattering operator is well-defined in the whole energy space $H^1(\mathbb{R}^2)$ for nonlinear Schr\"{o}dinger equation
with exponential nonlinearity $(e^{\lambda|u|^2}-1)u$, where $0<\lambda<4\pi$.
\end{abstract}

\section{Introduction}

We consider the Cauchy problem for the following nonlinear Schr\"{o}dinger equation
\begin{equation}\label{1.1}
iu_t+\triangle u=f(u),
\end{equation}
\begin{equation}\label{1.1a}
f(u):= (e^{\lambda|u|^2}-1)u,
\end{equation}
in two spatial dimensions with initial data $u_0\in H^1$ and $0<\lambda<4\pi$.
Solutions of the above problem satisfy the conservation of mass and Hamiltonian
\begin{align*}
\begin{aligned}
M(u;t)&:=\int_{\mathbb{R}^2}|u|^2dx=M(u_0), \\
H(u;t)&:=\int_{\mathbb{R}^2}(|\nabla u|^2+F(u))dx=H(u_0),
\end{aligned}
\end{align*}
where
\begin{equation*}
F(u)=\frac{1}{\lambda}(e^{\lambda|u|^2}-\lambda|u|^2-1).
\end{equation*}

Nakamura and Ozawa\cite{Nakamura} showed the existence and
uniqueness of the scattering operator of \eqref{1.1} with \eqref{1.1a}.
Then, Wang\cite{Wang} proved the smoothness of this scattering operator. However, both of these
results are based on the assumption of small initial data $u_0$. In
this paper, we  remove this assumption and show that for arbitrary
 initial data $u_0\in H^1(\mathbb{R}^2)$ and $H(u_0)\leq 1$, the scattering
operator is always well-defined.

Wang et al.\cite{W-H} proved the energy scattering theory of \eqref{1.1} with $f(u)=(e^{\lambda|u|^2}-1-\lambda|u|^2-\frac{\lambda^2}{2}|u|^4)u$, where $\lambda\in\mathbb{R}$ and the spatial dimension $n=1$. Ibrahim et al.\cite{Naka2009} showed the existence and asymptotic completeness of the wave operators for \eqref{1.1} with $f(u)=(e^{\lambda|u|^2}-1-\lambda|u|^2)u$ when the spatial
 dimensions $n=2$, $\lambda=4\pi$ and $H(u_0)\leq 1$.
Under the same assumptions as \cite{Naka2009},  Colliander et al.\cite{Coll2} proved the global well-posedness of \eqref{1.1} with \eqref{1.1a}:
 \begin{Theorem}
Assume that $u_0\in H^1(\mathbb{R}^2)$, $H(u_0)\leq 1$ and $\lambda=4\pi$.
Then  problem \eqref{1.1} with \eqref{1.1a} has a unique global solution $u$ in the class $C(\mathbb{R},H^1(\mathbb{R}^2))$.
\end{Theorem}
\begin{Remark}
In fact, by the proof in \cite{Coll2}, the global well-posedness of \eqref{1.1} with \eqref{1.1a} is also true for $0<\lambda\leq4\pi$.
\end{Remark}

 In this paper, we further study the scattering of this problem.
Note that $f(u)=(e^{\lambda|u|^2}-1)u=\sum_{k=1}^\infty\frac{\lambda^k}{k!}|u|^{2k}u$.
Nakanishi\cite{Naka99} proved the existence of the scattering operators in the whole energy space $H^1(\mathbb{R}^2)$ for \eqref{1.1} with $f(u)=|u|^pu$ when $p>2$.  Then, Killip et al.\cite{B-K-T-V} and Dodson\cite{Dodson} proved the existence of the scattering operators in $L^2(\mathbb{R}^2)$ for\eqref{1.1} with $f(u)=|u|^2u$.  Inspired by these two works, we use the concentration compactness method, which was introduced by Kenig and Merle in \cite{K-M},
 to prove the existence of the scattering operators for \eqref{1.1} with \eqref{1.1a}.

For convenience, we write \eqref{1.1} and  \eqref{1.1a} together, i.e.
\begin{equation}\label{1.1b}
iu_t+\triangle u=f(u):=(e^{\lambda|u|^2}-1)u,\ \ u(0,x)=u_0,
\end{equation}
where $u_0\in H^1(\mathbb{R}^2)$ and $0<\lambda<4\pi$.
Our main result is:
 \begin{Theorem} \label{th1.2}
 Assume that the  initial date $u_0\in H^1(\mathbb{R}^2)$, $H(u_0)\leq 1$ and $0<\lambda<4\pi$. Let $u$ be a global solution of \eqref{1.1b}. Then
\begin{equation}\label{1.2}
\| u \|_{L_{t,x}^4(\mathbb{R}\times \mathbb{R}^2)}<\infty.
\end{equation}
\end{Theorem}
In Section 2, Lemma \ref{Lem2.06} will show us that Theorem \ref{th1.2} implies the following scattering result:
\begin{Theorem} \label{th1.1}
Assume that the initial date $u_0\in H^1(\mathbb{R}^2)$,  $H(u_0)\leq 1$ and $0<\lambda<4\pi$.
 Then the solution of  \eqref{1.1b} is scattering.
\end{Theorem}
We will prove Theorem \ref{th1.2} by contradiction in Section 5. In Section 2, we give some nonlinear estimates.
In Section 3, we prove the stability of solutions. In Section 4, we  give a new profile decomposition for $H^1$ sequence
which will be used to prove concentration compactness.

Now, we introduce some notations:
\begin{gather*}
\langle x\rangle=\sqrt{1+|x|^2},\ \ \langle\cdot,\cdot\rangle\ is\ the\ inner\ product\ in\ L^2(\mathbb{R}^2) ,\\
 G(u):=\bar{u}f(u)-F(u)=e^{\lambda|u|^2}|u|^2-\frac{1}{\lambda}(e^{\lambda|u|^2}-1)=\sum\nolimits_{k=1}^{\infty}\frac{k\lambda^k|u|^{2k+2}}{(k+1)!},\\
E=E(u;t):=M(u;t)+H(u;t).
\end{gather*}
We define
\begin{align*}
\|u\|_{H^s_q(\mathbb{R}^2)}
:=\|(I-\triangle)^{s/2} u\|_{L^q(\mathbb{R}^2)},\ \
\|u\|_{\dot{H}^s_q(\mathbb{R}^2)}
:=\|(-\triangle)^{s/2} u\|_{L^q(\mathbb{R}^2)}.
\end{align*}
For Banach space $X=H^s_q(\mathbb{R}^2)$, $\dot{H}^s_q(\mathbb{R}^2)$ or $L^q(\mathbb{R}^2)$, we denote
\begin{align*}
\|u\|_{L^p(\mathbb{R};X)}
:=\left(\int_{\mathbb{R}}\|u(t)\|_X^p dt\right)^{1/p},
\end{align*}
When $q$ =$r$, we abbreviate $L_t^qL_x^r$ as $L_{t,x}^q$. When $q$ or $r$ are infinity, or when the domain
$\mathbb{R}\times\mathbb{R}^2$ is replaced by $I\times\mathbb{R}^2$, we make the usual modifications. Specially,
we denote
\begin{equation*}
S(u):=\|u\|^4_{L_{t,x}^4(\mathbb{R}\times\mathbb{R}^2)}.
\end{equation*}
If $t_0\in \mathbb{R}$, we  split $S(u)=S_{\leq t_0}(u)+S_{\geq t_0}(u)$, where
\begin{equation*}
S_{\leq t_0}(u):=\int_{-\infty}^{t_0}\int_{\mathbb{R}^2}|u(t,x)|^4dxdt
\end{equation*}
and
\begin{equation*}
S_{\geq t_0}(u):=\int_{t_0}^{+\infty}\int_{\mathbb{R}^2}|u(t,x)|^4dxdt.
\end{equation*}

For any two Banach spaces $X$ and $Y$, $\|\cdot\|_{X\cap Y}:=\max \{\|\cdot\|_{X},\|\cdot\|_{Y}\}$.
$C$ denotes positive constant. If $C$ depends upon some
parameters, such as $\lambda$, we will indicate this with $C(\lambda)$.
\begin{Remark}
Note that $0<\lambda<4\pi$ in Theorem \ref{th1.2}, we only need to prove the result for $0<\lambda<4(1-4\varepsilon)\pi$, $\varepsilon\in(0,1/8)$.
Hence, we always suppose that $0<\lambda<4(1-4\varepsilon)\pi$ in the context.

Moreover, we always suppose that the initial date $u_0$ of \eqref{1.1b} satisfies $u_0\in H^1(\mathbb{R}^2)$ and  $H(u_0)\leq 1$.
\end{Remark}

\section{Nonlinear Estimates}
In order to estimate \eqref{1.1a}, we need the following Trudinger inequality.
\begin{Lemma}\label{Lem2.03}{\rm(\cite{AT})}
Let $\lambda\in[0,4\pi)$. Then for all $u\in H^1(\mathbb{R}^2)$ satisfying $\|\nabla u\|_{L^2(\mathbb{R}^2)}\leq 1$, we have
\begin{equation*}
\int_{\mathbb{R}^2}\left( e^{\lambda|u|^2}-1\right) dx\leq  C(\lambda)\|u\|^2_{L^2(\mathbb{R}^2)}.
\end{equation*}
\end{Lemma}

 Note that for $ \forall\alpha\geq1$,
$$
(e^{\lambda|u|^2}-1)^\alpha\leq e^{\lambda\alpha|u|^2}-1.
$$
By Lemma \ref{Lem2.03} and H\"{o}lder inequlity, for $\lambda\in(0,4\pi)$ and $\forall\beta\geq0$, we have
\begin{equation}\label{2.03}
\begin{split}
\int_{\mathbb{R}^2}\left( e^{\lambda|u|^2}-1\right)|u|^\beta dx
\leq &\|e^{\lambda|u|^2}-1\|_{L_x^{1/(1-\varepsilon)}(\mathbb{R}^2)}\|u\|^\beta_{L_x^{\beta/\varepsilon}(\mathbb{R}^2)}\\
\leq &\|u\|^2_{L^2(\mathbb{R}^2)}\|u\|^\beta_{H^1(\mathbb{R}^2)}
\leq  C(\lambda, \beta)\|u\|^2_{L^2(\mathbb{R}^2)}
\end{split}
\end{equation}
and thus
\begin{equation}\label{2.017}
\begin{split}
\int_{\mathbb{R}^2}\left( e^{\lambda|u|^2}-\lambda|u|^2-1\right)dx
\leq\lambda\int_{\mathbb{R}^2}\left( e^{\lambda|u|^2}-1\right)|u|^2 dx\leq C(\lambda)\|u\|^2_{L^2(\mathbb{R}^2)}.
\end{split}
\end{equation}
\begin{Lemma}\label{Lem2.1}{\rm(Strichartz estimates)}
For $s=0\ or\ 1$,
$$
2\leq r,p<\infty,\ \ \frac{1}{\gamma(p)}+\frac{1}{p}=\frac{1}{2},\ \  \frac{1}{p}+\frac{1}{p'}=1,
$$
(the pairs $(\gamma(p), p)$ were called admissible pairs) we have
\begin{gather}
\|e^{i(t-t_0)\triangle}u(t_0)\|_{L^{\gamma(p)}(\mathbb{R};H^s_p)}\leq C\|u(t_0)\|_{H^s(\mathbb{R}^2)},\\
\|\int_{t_0}^te^{i(t-\tau)\triangle}f(u(\tau))d\tau\|_{L^{\gamma(p)}(\mathbb{R};H^s_p)}
\leq C\|f(u)\|_{L^{\gamma(r)'}(\mathbb{R};H^s_{r'})}.
\end{gather}
\end{Lemma}
\begin{Lemma}\label{B-Lem2.3}{\rm(Proposition 2.3, \cite{W-H})}
Let $1<r<p<\infty$ be fixed indices. Then for any $q\in[p,\infty)$
\begin{equation*}
\|u\|_{L^{q}(\mathbb{R}^2)}\leq C(p,r)q^{1/r'+p/rq}\|u\|^{p/q}_{L^{p}(\mathbb{R}^2)}\|u\|^{1-p/q}_{\dot{H}^{2/r}_r(\mathbb{R}^2)}.
\end{equation*}
\end{Lemma}

As is shown in \cite{Gini85} and \cite{Naka99} , to obtain the scattering result, it  suffices to show that
any finite energy solution has a finite global space-time norm. So, if Theorem \ref{th1.2} is true,
 we only need to prove the following theorem.
\begin{Lemma}\label{Lem2.06}{\rm(Theorem \ref{th1.2} implies Theorem \ref{th1.1})}
Let u be a global solution of \eqref{1.1b},
 $\|u\|_{L_{t,x}^4(\mathbb{R}\times\mathbb{R}^2)}<\infty$. Then, for all admissible pairs, we have
\begin{equation}\label{2.10}
\|u\|_{L^{\gamma(p)}(\mathbb{R};H^1_{p})}< \infty.
\end{equation}
\end{Lemma}

\begin{proof}
Defining  $X=L^{2/(1-2\varepsilon)}(I;H^1_{1/\varepsilon})$,
$Y=L^4(I;H^1_4)$,
by Strichartz estimates, \eqref{2.03} and \eqref{2.017},
\begin{equation}\label{2.038}
\begin{split}
&\|u\|_{X\cap Y}\\
\leq& C\|u(S)\|_{H^1(\mathbb{R}^2)}+C\|\lambda|u|^2u\|_{L^{4/3}(I;H^1_{4/3})}\\
&+C\|(e^{\lambda|u|^2}-\lambda|u|^2-1)u\|_{L^{2/(1+2\varepsilon)}(I;H^1_{1/(1-\varepsilon)})}\\
\leq& C(E)+\lambda C\|u\|^2_{L^4_{t,x}(I\times\mathbb{R}^2))}\|u\|_Y\\
&+\lambda C\|u\|^{8\varepsilon}_{L^4_{t,x}(I\times\mathbb{R}^2))}\|u\|_X
\|(e^{\lambda|u|^2}-\lambda|u|^2-1)u^{-8\varepsilon}\|_{L_t^\infty L_x^{1/(1-4\varepsilon)}(I\times\mathbb{R}^2)}\\
\leq& C(E)+ C(E)(\|u\|^2_{L^4_{t,x}(I\times\mathbb{R}^2))}+\|u\|^{8\varepsilon}_{L^4_{t,x}(I\times\mathbb{R}^2))})
\|u\|_{X\cap Y}.
\end{split}
\end{equation}
Using the same way as in Bourgain \cite{Bourgain99}, one can split $\mathbb{R}$ into finitely many pairwise disjoint intervals
\begin{equation}
\mathbb{R}=\bigcup\nolimits_{j=1}^JI_j,\ \ \ \|u\|_{L_{t,x}^4(I_j\times\mathbb{R}^2)}\leq\eta,\ \ \ C(E)(\eta^2+\eta^{8\varepsilon})\leq 1/2.
\end{equation}
By \eqref{2.038},
\begin{equation}
\|u\|_{L^{2/(1-2\varepsilon)}(I_j;H^1_{1/\varepsilon})\cap L^4(I_j;H^1_4)}\leq C(E) .
\end{equation}
As $\varepsilon\in(0,1/8)$ can be chosen small arbitrarily, by interpolation,
\begin{equation}
\|u\|_{L^{\gamma(p)}(I_j;H^1_{p})}\leq C(E),
\end{equation}
for all admissible pairs and $j=1,2,\cdots, J$. The desire result follows.
\end{proof}

\section{Stability}

\begin{Lemma} \label{B-Lem1.4}{\rm(Stability)}
For any $A>0$ and $\sigma>0$, there exists $\delta>0$ with the following property: if $u:I\times\mathbb{R}^2\rightarrow \mathbb{C}$ satisfies
$\|u\|_{L^4_{t,x}(I\times\mathbb{R}^2)}\leq A$ and approximately solves \eqref{1.1b} in the sense that
\begin{equation}\label{B-1.4}
\|\int_{t_0}^te^{i(t-\tau)\triangle}(iu_t+\triangle u-f(u))(\tau)d\tau\|_{L_{t,x}^4\cap L_t^{2/(1-2\varepsilon)}L_{x}^{1/\varepsilon}(I\times\mathbb{R}^2)}\leq \delta;
\end{equation}
and $v_0\in H^1(\mathbb{R}^2)$ satifies $H(v_0)\leq 1$ and $\|e^{i(t-t_0)\triangle}(u(t_0)-v_0)\|_{L^4_{t,x}(I\times\mathbb{R}^2)}\leq \delta$ for some $t_0\in I$,
then there exists a solution $v:I\times\mathbb{R}^2\rightarrow \mathbb{C}$ to \eqref{1.1b} with $v(t_0)=v_0$ such that
$ \|u-v\|_{L^4_{t,x}(I\times\mathbb{R}^2)}\leq \sigma$.
\end{Lemma}
\begin{proof}
Let $v:\mathbb{R}\times\mathbb{R}^2\rightarrow \mathbb{C}$ be the global
solution with the initial data $v(t_0)=v_0$. Denote $v=u+w$ on the interval $I$, then
\begin{equation}\label{B-1.5}
iw_t+\triangle w=(f(u+w)-f(u))-(iu_t+\triangle u-f(u))
\end{equation}
and $\|e^{i(t-t_0)\triangle}w(t_0)\|_{L_{t,x}^4(I\times\mathbb{R}^2)}\leq \delta$. Let $X=L_{t,x}^4\cap L_t^{2/(1-2\varepsilon)}L_x^{1/\varepsilon}(I\times\mathbb{R}^2)$, by Strichartz estimates, \eqref{B-1.4} and triangle inequality, we have
\begin{equation}\label{B-1.6}
\begin{split}
\|w\|_X&\lesssim \delta^{2\varepsilon/(1-2\varepsilon)}+\|\int_{t_0}^te^{i(t-\tau)\triangle}(f(u+w)-f(u))(\tau)d\tau\|_X+ \delta\\
&\lesssim2\delta^{2\varepsilon/(1-2\varepsilon)}+(\|u\|^2_{L_{t,x}^4(I\times\mathbb{R}^2)}+\|u+w\|^2_{L_{t,x}^4(I\times\mathbb{R}^2)})
\|w\|_{L_{t,x}^4(I''\times\mathbb{R}^2)}\\
&+(\|u\|^{8\varepsilon}_{L_{t,x}^4(I\times\mathbb{R}^2)}+
\|u+w\|^{8\varepsilon}_{L_{t,x}^4(I\times\mathbb{R}^2)})\|w\|_{L_t^{2/(1-2\varepsilon)}L_{x}^{1/\varepsilon}(I\times\mathbb{R}^2)}\\
&\lesssim 2\delta^{2\varepsilon/(1-2\varepsilon)} +(2A^2+2A^{8\varepsilon})\|w\|_X+\|w\|_X^{1+8\varepsilon}+\|w\|_X^3.
\end{split}
\end{equation}
When $A$ and $\delta=\delta(\sigma)$ both are sufficiently small,  standard continuity argument gives $\|w\|_X\leq \sigma$.
When $A$ is large, we only need to subdivide the time interval $I$ and then the result follows by an iterate process.
\end{proof}


\section{Linear Profile Decomposition}

In this section, we will give the linear profile decomposition for Schr\"{o}dinger equation in $H^1(\mathbb{R}^2)$.
First, we give some definitions and lemmas.
\begin{Definition}\label{B-Def 1.2}{\rm(Symmetry group, \cite{B-T-M-Z})}
For any phase $\theta\in\mathbb{R}/2\pi\mathbb{Z}$, position $x_0\in\mathbb{R}^2$, frequency $\xi_0\in\mathbb{R}^2$,
and scaling parameter $\lambda>0$, we define the unitary transformation $g_{\theta,\xi_0,x_0,\lambda}:L_x^2(\mathbb{R}^2)\rightarrow L_x^2(\mathbb{R}^2)$ by the formula
\begin{equation*}
g_{\theta,\xi_0,x_0,\lambda}f(x):=\frac{1}{\lambda}e^{i\theta}e^{ix\cdot\xi_0}f(\frac{x-x_0}{\lambda}).
\end{equation*}
We let $G$ be the collection of such transformations; this is a group with identity $g_{0,0,0,1}$,
inverse $g^{-1}_{\theta,\xi_0,x_0,\lambda}=g_{-\theta-x_0\xi_0,-\lambda\xi_0,-x_0/\lambda,1/\lambda}$ and group law
\begin{equation*}
g_{\theta,\xi_0,x_0,\lambda}g_{\theta',\xi'_0,x'_0,\lambda'}=g_{\theta+\theta'-x_0\xi'_0/\lambda,\xi_0+\xi'_0/\lambda,x_0+\lambda x'_0,\lambda\lambda'}.
\end{equation*}

We let $G\backslash L_x^2(\mathbb{R}^2)$
be the modulo space of $G$-orbits $Gf:=\{gf:g\in G\}$ of $L_x^2(\mathbb{R}^2)$, endowed with the usual quotient topology.
If $u:I\times\mathbb{R}^2\rightarrow\mathbb{C}$ is a function, we define $T_{g_{\theta,\xi_0,x_0,\lambda}}u:\lambda^2I\times\mathbb{R}^2\rightarrow\mathbb{C}$ where $\lambda^2I:=\{\lambda^2t:t\in I\}$
by the formula
\begin{equation*}
(T_{g_{\theta,\xi_0,x_0,\lambda}}u)(t,x)
:=\frac{1}{\lambda}e^{i\theta}e^{ix\cdot\xi_0}e^{-it|\xi_0|^2}u(\frac{t}{\lambda^2},\frac{x-x_0-2\xi_0t}{\lambda}),
\end{equation*}
or equivalently
\begin{equation*}
(T_{g_{\theta,\xi_0,x_0,\lambda}}u)(t)=g_{\theta-t|\xi_0|^2,\xi_0,x_0+2\xi_0t,\lambda}(u(\frac{t}{\lambda^2})).
\end{equation*}
\end{Definition}
If $g\in G$, we can easily prove that
 $M(T_gu)=M(u)$ and $S(T_gu)=S(u)$.
\begin{Definition}\label{B-Def 1.3}{\rm(Enlarged group, \cite{B-T-M-Z})}
For any phase $\theta\in\mathbb{R}/2\pi\mathbb{Z}$, position $x_0\in\mathbb{R}^2$, frequency $\xi_0\in\mathbb{R}^2$,
scaling parameter $\lambda>0$, and time $t_0$, we define the unitary transformation $g_{\theta,\xi_0,x_0,\lambda,t_0}:L_x^2(\mathbb{R}^2)\rightarrow L_x^2(\mathbb{R}^2)$ by the formula
\begin{equation*}
g_{\theta,\xi_0,x_0,\lambda,t_0}=g_{\theta,\xi_0,x_0,\lambda}e^{it_0\triangle},
\end{equation*}
or in other words
\begin{equation*}
g_{\theta,\xi_0,x_0,\lambda,t_0}f(x):=\frac{1}{\lambda}e^{i\theta}e^{ix\cdot\xi_0}(e^{it_0\triangle}f)(\frac{x-x_0}{\lambda}).
\end{equation*}
Let $G'$ be the collection of such transformations. We also let $G'$ act on global space-time function $u:\mathbb{R}\times\mathbb{R}^2\rightarrow\mathbb{C}$ by defining
\begin{equation*}
T_{g_{\theta,\xi_0,x_0,\lambda,t_0}}u(t,x):=
\frac{1}{\lambda}e^{i\theta}e^{ix\cdot\xi_0}e^{-it|\xi_0|^2}(e^{it_0\triangle}u)(\frac{t}{\lambda^2},\frac{x-x_0-2\xi_0t}{\lambda}),
\end{equation*}
or equivalently
\begin{equation*}
(T_{g_{\theta,\xi_0,x_0,\lambda,t_0}}u)(t)=g_{\theta-t|\xi_0|^2,\xi_0,x_0+2\xi_0t,\lambda,t_0}(u(\frac{t}{\lambda^2})).
\end{equation*}
\end{Definition}

\begin{Lemma} \label{B-Lem1.5}{\rm(Linear profiles for $L^2$ sequence, \cite{B-M-V})}
Let $u_n$ be a bounded sequence in $L_x^2(\mathbb{R}^2)$. Then (after passing to a subsequence if necessary) there exists a family $\phi^{(j)}$, $j=1,2,\cdots$ of functions in  $L_x^2(\mathbb{R}^2)$ and group elements $g_n^{(j)}\in G'$ for $j,n=1,2,\cdots$ such that we have the decomposition
\begin{equation}\label{B-1.7}
u_n=\sum_{j=1}^lg_n^{(j)}\phi^{(j)}+w_n^{(l)}
\end{equation}
for all $l=1,2,\cdots$; here $w_n^{(l)}\in L_x^2(\mathbb{R}^2)$ is such that its linear evolution has asymptotically vanishing scattering size:
\begin{equation}\label{B-1.8}
\lim_{l\rightarrow\infty}\limsup_{n\rightarrow\infty}S(e^{it\triangle}w_n^{(l)})=0.
\end{equation}
Moreover, for any $j\neq j'$,
\begin{equation}\label{N-2}
\frac{\lambda_n^{(j')}}{\lambda_n^{(j)}}+\frac{\lambda_n^{(j)}}{\lambda_n^{(j')}}
+\lambda_n^{(j)}\lambda_n^{(j')}|\xi_n^{(j)}-\xi_n^{(j')}|^2
+\frac{|x_n^{(j)}-x_n^{(j')}|^2}{\lambda_n^{(j)}\lambda_n^{(j')}}
+\frac{|t_n^{(j)}(\lambda_n^{(j)})^2-t_n^{(j')}(\lambda_n^{(j')})^2|}{\lambda_n^{(j)}\lambda_n^{(j')}}
\rightarrow \infty.
\end{equation}
Furthermore, for any $l\geq 1$ we have the mass decoupling property
\begin{equation}\label{B-1.9}
\lim_{n\rightarrow\infty}[M(u_n)-\sum_{j=1}^lM(\phi^{(j)})-M(w_n^{(l)})]=0;
\end{equation}
for any $j\leq l$, we have
\begin{equation}\label{N-3}
 (g_n^{(j)})^{-1}w_n^{(l)}\rightharpoonup 0, \ \ weakly\ in\ L_x^2(\mathbb{R}^2).
\end{equation}

\end{Lemma}
\begin{Remark}
If the orthogonal condition \eqref{N-2} holds, then (see \cite{B-M-V})
\begin{align*}
\begin{aligned}
&\lim_{n\rightarrow\infty}\langle g_{n}^{(j)}\phi^{(j)}, g_{n}^{(j')}\phi^{(j')}\rangle_{L^2(\mathbb{R}^2)}=0, \ \ j\neq j',\\
&\lim_{n\rightarrow\infty}\langle g_{n}^{(j)}\phi^{(j)}, w_{n}^{(l)}\rangle_{L^2(\mathbb{R}^2)}=0.
\end{aligned}
\end{align*}
Moreover, if $v^{(j)},v^{(j')}\in L_{t,x}^4(\mathbb{R}\times\mathbb{R}^2)$, then (see \cite{B-B-G}, \cite{B-M-V}), for any $0<\theta<1$
\begin{equation}\label{B-1.10}
\lim_{n\rightarrow\infty}\||T_{g_n^{(j)}}v^{(j)}|^{1-\theta}|T_{g^{(j')}_n}v^{(j')}|^\theta\|_{L_{t,x}^4(\mathbb{R}\times\mathbb{R}^2)}=0;
\end{equation}
if $v^{(1)},\cdots,v^{(l)}\in L_{t,x}^4(\mathbb{R}\times\mathbb{R}^2)$, then (see \cite{B-B-V}, Lemma 5.5)
\begin{equation}\label{B-1.11}
\lim_{n\rightarrow\infty}S(\sum_{j=1}^lg_n^{(j)}v^{(j)})\leq\sum_{j=1}^lS(v^{(j)}).
\end{equation}
\end{Remark}
\begin{Remark}
As each linear profile $\phi^{(j)}$ in Lemma \ref{B-Lem1.5} is constructed in the sense that
\begin{equation*}
e^{it_n^{(j)}\triangle}\{e^{i\xi_n^{(j)}\cdot x}\lambda_n^{(j)}u_n(\lambda_n^{(j)}x)\}(x+x_n^{(j)})\rightharpoonup \phi^{(j)}
\end{equation*}
weakly in $L_x^2(\mathbb{R}^2)$(see \cite{B-M-V}),
after passing to a  subsequence in $n$, rearrangement, translation, and refining $\phi^{(j)}$ accordingly,
we may assume that the parameters satisfy the following:

1) $t_n^{(j)}\rightarrow \pm\infty$ as $n\rightarrow\infty$, or $t_n^{(j)}\equiv 0$ for all $n,j$;

2)  $\lambda_n^{(j)}\rightarrow 0$ or $\infty$ as $n\rightarrow\infty$, or $\lambda_n^{(j)}\equiv 1$ for all $n,j$;

3)  $|\xi_n^{(j)}|\rightarrow \infty$ as $n\rightarrow\infty$, or $\xi_n^{(j)}\equiv \xi^{(j)}$ with $|\xi^{(j)}|< \infty$.

4)  When $\lambda_n^{(j)}\equiv 1$, $\xi_n^{(j)}\equiv \xi^{(j)}$ and $|\xi^{(j)}|< \infty$, we can let $\xi^{(j)}\equiv0$.
\end{Remark}

Our main result in this section is the following lemma:
\begin{Lemma} \label{N2-Lem1}{\rm(Linear profiles for $H^1$ sequence)}
Let $u_n$ be a bounded sequence in $H^1(\mathbb{R}^2)$. Then up to a subsequence, for any $J\geq1$,
there exists a sequence $\phi_{\alpha}$ in $H^1(\mathbb{R}^2)$
and a sequence of group elements $g_{n\alpha}=g_{\theta_{n\alpha},\xi_{n\alpha},x_{n\alpha},\lambda_{n\alpha},t_{n\alpha}}\in G'$ such that
\begin{equation}
\begin{split}\label{N2-21}
u_n=&\sum_{\alpha=1}^Jg_{n\alpha}\phi_\alpha+R(n,J).
\end{split}
\end{equation}
Here, for each $\alpha$, $\lambda_{n\alpha}$ and $\xi _{n\alpha}$ must satisfy
\begin{equation}\label{N2-20}
 \lambda_{n\alpha}\equiv1 \ and\  \xi _{n\alpha}\equiv0, or \lambda_{n\alpha}\rightarrow\infty;
\end{equation}
$R(n,J)\in H^1(\mathbb{R}^2)$ is such that
\begin{equation}\label{N2-1}
\lim_{J\rightarrow\infty}\limsup_{n\rightarrow\infty}S(e^{it\triangle}R(n,J))=0.
\end{equation}
Moreover, for any $\alpha\neq \alpha'$, one has the same orthogonal conditions as \eqref{N-2}.
For any $J\geq 1$, one has the following decoupling properties
\begin{equation}\label{N2-2}
\lim_{n\rightarrow\infty}\{\|u_n\|^2_{L^2(\mathbb{R}^2)}-\sum_{\alpha=1}^J
\|\phi_\alpha\|^2_{L^2(\mathbb{R}^2)}-\|R(n,J)\|^2_{L^2(\mathbb{R}^2)}\}=0,
\end{equation}
\begin{equation}\label{N2-3}
\lim_{n\rightarrow\infty}\{\|u_n\|^2_{\dot{H}^1(\mathbb{R}^2)}-\sum_{\alpha=1}^J
\|g_{n\alpha}\phi_\alpha\|^2_{\dot{H}^1(\mathbb{R}^2)}-\|R(n,J)\|^2_{\dot{H}^1(\mathbb{R}^2)}\}=0,
\end{equation}
\begin{equation}\label{N2-4}
\lim_{J\rightarrow\infty}\limsup_{n\rightarrow\infty}\{H(u_n)-
\sum_{\alpha=1}^J
H(g_{n\alpha}\phi_\alpha)-H(R(n,J))
\}=0.
\end{equation}

\end{Lemma}
\begin{proof}
Let
\begin{equation*}
\Box_k=\mathcal {F}^{-1}\chi_k\mathcal {F},
\ \ \ \chi_k=\begin{cases}
1& \text{$2^{k-1}<|\xi|\leq 2^{k}$};\\
0& \text{$else$.}
\end{cases}
\end{equation*}
Then, we have
\begin{equation*}
u_n=\sum_{k=-\infty}^{+\infty}\Box_ku_n:=\sum_{|k|\leq N}\Box_ku_n+R_N,
\end{equation*}
and
\begin{align*}
\begin{aligned}
&\|u_n\|^2_{L^2(\mathbb{R}^2)}=\sum_{|k|\leq N}\|\Box_ku_n\|^2_{L^2(\mathbb{R}^2)}+\|R_N\|^2_{L^2(\mathbb{R}^2)}, \\
&\|u_n\|^2_{\dot{H}^1(\mathbb{R}^2)}=\sum_{|k|\leq N}\|\Box_ku_n\|^2_{\dot{H}^1(\mathbb{R}^2)}+\|R_N\|^2_{\dot{H}^1(\mathbb{R}^2)}, \\
&\lim_{N\rightarrow\infty}\limsup_{n\rightarrow\infty}\|R_N\|_{L^2(\mathbb{R}^2)}=0.
\end{aligned}
\end{align*}

By Lemma \ref{B-Lem1.5}, after passing to a subsequence if necessary, we can obtain
\begin{equation}\label{B-1.13}
\Box_ku_n=\sum_{j=1}^{l_k}g_{nk}^{(j)}\phi_k^{(j)}+w_{nk}^{(l_k)}
\end{equation}
with the stated properties 1)-4) and \eqref{B-1.7}-\eqref{N-3}.
Denote
\begin{align*}
\begin{aligned}
&\Lambda_{1,0}=\{(k,j)\ |\  |k|\leq N, 1\leq j\leq l_k, \lambda_{nk}^{(j)}\equiv1, \xi_{nk}^{(j)}\equiv 0 \};\\
&\Lambda_{1,\infty}=\{(k,j)\ |\   |k|\leq N, 1\leq j\leq l_k, \lambda_{nk}^{(j)}\equiv1, |\xi_{nk}^{(j)}|\rightarrow\infty \};\\
&\Lambda_{0}=\{(k,j)\ |\   |k|\leq N, 1\leq j\leq l_k, \lambda_{nk}^{(j)}\rightarrow 0 \};\\
&\Lambda_{\infty,0}=\{(k,j)\ |\   |k|\leq N, 1\leq j\leq l_k, \lambda_{nk}^{(j)}\rightarrow\infty, \xi_{nk}^{(j)}\equiv \xi_{k}^{(j)},
|\xi_{k}^{(j)}|<2^{k-1} \};\\
&\Lambda_{\infty,1}=\{(k,j)\ |\   |k|\leq N, 1\leq j\leq l_k, \lambda_{nk}^{(j)}\rightarrow\infty, \xi_{nk}^{(j)}\equiv \xi_{k}^{(j)},
|\xi_{k}^{(j)}|\in [2^{k-1}, 2^k]\};\\
&\Lambda_{\infty,\infty}=\{(k,j)\ |\   |k|\leq N, 1\leq j\leq l_k, \lambda_{nk}^{(j)}\rightarrow\infty, |\xi_{nk}^{(j)}|\rightarrow\infty  \ or\ \xi_{nk}^{(j)}\equiv \xi_{k}^{(j)}, |\xi_{k}^{(j)}|>2^{k}\}.
\end{aligned}
\end{align*}
\noindent\textbf{Step 1.} We prove that
\begin{equation}
\begin{split}\label{N-110}
u_n=&\sum_{(k,j)\in\Lambda_{1,0}\cup\Lambda_{\infty,1}}g_{nk}^{(j)}\phi_k^{(j)}+R
\end{split}
\end{equation}
with
$g_{nk}^{(j)}\phi_k^{(j)}=\Box_kg_{nk}^{(j)}\phi_k^{(j)}$ and for each fixed $N$,
\begin{equation}\label{N-217}
\lim_{n\rightarrow\infty}\{\|u_n\|^2_{L^2(\mathbb{R}^2)}-\sum_{(k,j)\in\Lambda_{1,0}\cup\Lambda_{\infty,1}}
\|\phi_k^{(j)}\|^2_{L^2(\mathbb{R}^2)}-\|R\|^2_{L^2(\mathbb{R}^2)}\}=0,
\end{equation}
\begin{equation}\label{N-218}
\lim_{n\rightarrow\infty}\{\|u_n\|^2_{\dot{H}^1(\mathbb{R}^2)}-\sum_{(k,j)\in\Lambda_{1,0}\cup\Lambda_{\infty,1}}
\|g_{nk}^{(j)}\phi_k^{(j)}\|^2_{\dot{H}^1(\mathbb{R}^2)}-\|R\|^2_{\dot{H}^1(\mathbb{R}^2)}\}=0,
\end{equation}
\begin{equation}\label{N-219}
\lim_{N\rightarrow\infty}\lim_{l_k\rightarrow\infty}\limsup_{n\rightarrow\infty}S(e^{it\triangle}R)=0,
\end{equation}
where
\begin{equation*}
R=R_N+R_w,\ \ R_w=\sum_{|k|\leq N}w_{nk}^{(l_k)}.
\end{equation*}

By \eqref{B-1.8} and $\lim_{N\rightarrow\infty}\limsup_{n\rightarrow\infty}\|R_N\|_{L^2(\mathbb{R}^2)}=0$, \eqref{N-219} holds obviously. For  \eqref{N-110},
we  prove it  by induction. For every $k$, suppose that
\begin{equation}\label{N-220}
\Box_ku_n=g_{nk}^{(1)}\phi_k^{(1)}+w_{nk}^{(1)}.
\end{equation}

Case 1. If $(k,1)\in \Lambda_{1,\infty}\cup\Lambda_{0}\cup\Lambda_{\infty,0}\cup\Lambda_{\infty,\infty}$, we have $\phi_k^{(1)}=0$.

In fact, by \eqref{N-220},
\begin{equation*}
\phi_k^{(1)}=(g_{nk}^{(1)})^{-1}\Box_ku_n-(g_{nk}^{(1)})^{-1}w_{nk}^{(1)}.
\end{equation*}
Thus
\begin{equation}
\begin{split}\label{N-221}
\|\phi_k^{(1)}\|^2_{L^2(\mathbb{R}^2)}&=\langle(g_{nk}^{(1)})^{-1}\Box_ku_n-(g_{nk}^{(1)})^{-1}w_{nk}^{(1)},\phi_k^{(1)}\rangle\\
&=\langle u_n,\Box_kg_{nk}^{(1)}\phi_k^{(1)}\rangle-\langle(g_{nk}^{(1)})^{-1}w_{nk}^{(1)},\phi_k^{(1)}\rangle.
\end{split}
\end{equation}

Using \eqref{N-3},
\begin{equation}
\langle(g_{nk}^{(1)})^{-1}w_{nk}^{(1)},\phi_k^{(1)}\rangle\rightarrow 0 \ as \ n\rightarrow\infty.
\end{equation}
By direct calculation,
\begin{equation}\label{N-1}
\begin{split}
&\Box_kg_{nk}^{(1)}\phi_k^{(1)}\\
=&\mathcal {F}^{-1}\chi_k(\xi)\lambda_{nk}^{(1)}e^{i\theta_{nk}^{(1)}}e^{-ix_{nk}^{(1)}(\xi+\xi_{nk}^{(1)})}
e^{-it_{nk}^{(1)}(\lambda_{nk}^{(1)})^2|\xi+\xi_{nk}^{(1)}|^2}\widehat{\phi_k^{(1)}}(\lambda_{nk}^{(1)}(\xi+\xi_{nk}^{(1)}))\\
=&\frac{1}{\lambda_{nk}^{(1)}}e^{i\theta_{nk}^{(1)}}
e^{ix\cdot\xi_{nk}^{(1)}}(\mathcal {F}^{-1}\chi_k(\frac{\xi}{\lambda_{nk}^{(1)}}+\xi_{nk}^{(1)})
e^{-it_{nk}^{(1)}|\xi|^2}\widehat{\phi_k^{(1)}})(\frac{x-x_{nk}^{(1)}}{\lambda_{nk}^{(1)}}).
\end{split}
\end{equation}
Let $  n\rightarrow\infty$, when $(k,1)\in \Lambda_{1,\infty}$,
\begin{equation}\label{N-5}
\begin{split}
\|\Box_kg_{nk}^{(1)}\phi_k^{(1)}\|^2_{L^2(\mathbb{R}^2)}
\leq \int_{2^{k-1}\leq|\xi+\xi_{nk}^{(1)}|\leq2^{k}}|\mathcal {F}\phi_k^{(1)}|^2d\xi
\rightarrow 0 ;
\end{split}
\end{equation}
when $(k,1)\in \Lambda_{0}$,
\begin{equation}
\begin{split}
\|\Box_kg_{nk}^{(1)}\phi_k^{(1)}\|^2_{L^2(\mathbb{R}^2)}
\leq \int_{\lambda_{nk}^{(1)}2^{k-1}\leq|\xi+\lambda_{nk}^{(1)}\xi_{nk}^{(1)}|\leq\lambda_{nk}^{(1)}2^{k}}|\mathcal {F}\phi_k^{(1)}|^2d\xi
\rightarrow 0;
\end{split}
\end{equation}
when $(k,1)\in\Lambda_{\infty,0}$,
\begin{equation}
\begin{split}
\|\Box_kg_{nk}^{(1)}\phi_k^{(1)}\|^2_{L^2(\mathbb{R}^2)}
\leq\int_{|\xi|\geq\lambda_{nk}^{(1)}(2^{k-1}-|\xi_{k}^{(1)}|)}|\mathcal {F}\phi_k^{(1)}|^2d\xi
\rightarrow 0;
\end{split}
\end{equation}
when $(k,1)\in\Lambda_{\infty,\infty}$,
\begin{equation}\label{N-6}
\begin{split}
\|\Box_kg_{nk}^{(1)}\phi_k^{(1)}\|^2_{L^2(\mathbb{R}^2)}
\leq\int_{|\xi|\geq\lambda_{nk}^{(1)}(|\xi_{nk}^{(1)}|-2^{k})}|\mathcal {F}\phi_k^{(1)}|^2d\xi
\rightarrow 0.
\end{split}
\end{equation}
By \eqref{N-221}-\eqref{N-6}, $\|\phi_k^{(1)}\|^2_{L^2(\mathbb{R}^2)}=0$ and thus $\phi_k^{(1)}=0$.

Case 2. If $(k,1)\in \Lambda_{1,0}\cup\Lambda_{\infty,1}$, we have
\begin{equation}\label{N2-40}
\|g_{nk}^{(1)}\phi_k^{(1)}-\Box_kg_{nk}^{(1)}\phi_k^{(1)}\|_{L^2(\mathbb{R}^2)}\rightarrow0 \ as \ n\rightarrow\infty.
\end{equation}

Let $\chi_{A_k^{(1)}}$ be the characteristic function of the set $A_k^{(1)}$ and $P_{A_k^{(1)}}=\mathcal {F}^{-1}\chi_{A_k^{(1)}}\mathcal {F}$, then
\begin{equation*}
g_{nk}^{(1)}\left(P_{A_k^{(1)}}\phi_k^{(1)}+P_{A_k^{(1)}}(g_{nk}^{(1)})^{-1}w_{nk}^{(1)}\right)=P_{\tilde{A}_k^{(1)}}
(g_{nk}^{(1)}\phi_k^{(1)}+w_{nk}^{(1)})
=P_{\tilde{A}_k^{(1)}}\Box_ku_n,
\end{equation*}
where
\begin{equation*}
P_{\tilde{A}_k^{(1)}}=\mathcal {F}^{-1}\chi_{A_k^{(1)}}(\lambda_{nk}^{(1)}(\xi-\xi_{nk}^{(1)}))\mathcal {F}.
\end{equation*}
Note that
\begin{equation*}
\langle P_{A_k^{(1)}}\phi_k^{(1)}, P_{A_k^{(1)}}(g_{nk}^{(1)})^{-1}w_{nk}^{(1)}\rangle
= \langle P_{A_k^{(1)}}\phi_k^{(1)}, (g_{nk}^{(1)})^{-1}w_{nk}^{(1)}\rangle\rightarrow0 \ as \ n\rightarrow\infty,
\end{equation*}
we have
\begin{equation}\label{N-223}
\lim_{n\rightarrow\infty}\left(\|P_{\tilde{A}_k^{(1)}}\Box_ku_n\|^2_{L^2(\mathbb{R}^2)}-\|P_{A_k^{(1)}}\phi_k^{(1)}\|^2_{L^2(\mathbb{R}^2)}
-\|P_{A_k^{(1)}}(g_{nk}^{(1)})^{-1}w_{nk}^{(1)}\|^2_{L^2(\mathbb{R}^2)}\right)=0.
\end{equation}

When $(k,1)\in \Lambda_{1,0}$, we have $P_{A_k^{(1)}}=P_{\tilde{A}_k^{(1)}}$. Choosing $A_k^{(1)}=\{\xi \ | \ |\xi| \leq2^{k-1}\ or\ |\xi| >2^{k}\}$, then
by \eqref{N-223}, $P_{A_k^{(1)}}\phi_k^{(1)}=0$, the desired result follows.

When $(k,1)\in \Lambda_{\infty,1}$ and $|\xi_{k}^{(1)}|\in (2^{k-1}, 2^k)$, we have
\begin{equation*}
\begin{split}
&\|g_{nk}^{(1)}\phi_k^{(1)}-\Box_kg_{nk}^{(1)}\phi_k^{(1)}\|^2_{L^2(\mathbb{R}^2)}\\
\leq&\int_{|\xi+\lambda_{nk}^{(1)}\xi_{k}^{(1)}|\leq\lambda_{nk}^{(1)}2^{k-1}}|\mathcal {F}\phi_k^{(1)}|^2d\xi
+\int_{|\xi+\lambda_{nk}^{(1)}\xi_{k}^{(1)}|\geq\lambda_{nk}^{(1)}2^{k}}|\mathcal {F}\phi_k^{(1)}|^2d\xi\\
\leq&\int_{|\xi|\geq\lambda_{nk}^{(1)}(|\xi_{k}^{(1)}|-2^{k-1})}|\mathcal {F}\phi_k^{(1)}|^2d\xi
+\int_{|\xi|\geq\lambda_{nk}^{(1)}(2^{k}-|\xi_{k}^{(1)}|)}|\mathcal {F}\phi_k^{(1)}|^2d\xi\\
\rightarrow & 0 \ as \ n\rightarrow\infty
\end{split}
\end{equation*}

When $(k,1)\in\Lambda_{\infty,1}$ and $|\xi_{k}^{(1)}|= 2^{k}$, we denote $\xi=(\xi_1,\xi_2)$ and $\xi_{k}^{(1)}=(\xi_{k1}^{(1)},\xi_{k2}^{(1)})$.
The line $\xi_2=-\frac{\xi_{k1}^{(1)}}{\xi_{k2}^{(1)}}\xi_1$ (when $\xi_{k2}^{(j)}=0$, we use the line $\xi_1=0$ instead) separates the frequency
space $L^2(\mathbb{R}^2)$ into two half-planes. We let $A_k^{(1)}$ to be the half-plane which  contains the point $\xi_{k}^{(1)}$,
then
\begin{equation*}
P_{\tilde{A}_k^{(1)}}\Box_ku_n
=\mathcal {F}^{-1}\chi_{A_k^{(1)}}(\lambda_{nk}^{(1)}(\xi-\xi_{k}^{(1)}))\chi_k\mathcal {F}u_n=0.
\end{equation*}
By \eqref{N-223}, we have $P_{A_k^{(1)}}\phi_k^{(1)}=0$.
Note that
\begin{equation*}
\begin{split}
&\lim_{n\rightarrow\infty}\|g_{nk}^{(1)}(1-P_{A_k^{(1)}})\phi_k^{(1)}-\Box_kg_{nk}^{(1)}\phi_k^{(j)}\|^2_{L^2(\mathbb{R}^2)}\\
=&\lim_{n\rightarrow\infty}\|[(1-\chi_{A_k^{(1)}})(\cdot)-\chi_k(\frac{\cdot}{\lambda_{nk}^{(1)}}+\xi_{k}^{(1)})
]\mathcal {F}\phi_k^{(1)}\|^2_{L^2(\mathbb{R}^2)}\\
\leq&\lim_{n\rightarrow\infty}
\int_{\xi\in \mathbb{R}^2\setminus (A_k^{(1)}\cup\{2^{k-1}\leq|\frac{\xi}{\lambda_{nk}^{(1)}}+\xi_{k}^{(1)}|\leq2^k\})}|\mathcal {F}\phi_k^{(1)}|^2d\xi\\
=&0,
\end{split}
\end{equation*}
\eqref{N2-40} holds.

When $(k,1)\in\Lambda_{\infty,1}$ and $|\xi_{k}^{(1)}|= 2^{k-1}$, let $A_k^{(1)}$ to be the half-plane which does NOT contain the point $\xi_{k}^{(1)}$, we can prove \eqref{N2-40} similarly as above.

By the proof above and absorbing the error into $w_{nk}^{(1)}$, we can suppose
$g_{nk}^{(1)}\phi_k^{(1)}=\Box_kg_{nk}^{(1)}\phi_k^{(1)}$ and  $(k,1)\in \Lambda_{1,0}\cup\Lambda_{\infty,1}$.
Denote $u_n^{(1)}=u_n-g_{nk}^{(1)}\phi_k^{(1)}$ and suppose
\begin{equation}\label{N-220}
\Box_ku_n^{(1)}=g_{nk}^{(2)}\phi_k^{(2)}+w_{nk}^{(2)}.
\end{equation}
Repeating the proof above, we can obtain  $g_{nk}^{(2)}\phi_k^{(2)}=\Box_kg_{nk}^{(2)}\phi_k^{(2)}$ and  $(k,2)\in \Lambda_{1,0}\cup\Lambda_{\infty,1}$. $\cdots\cdots$, by induction, we obtain \eqref{N-110}.

 By the orthogonal condition \eqref{N-2},
following the proof in \cite{B-M-V}, we can obtain that for fix $k$ and $\forall j\neq j'$,
\begin{align*}
\begin{aligned}
&\lim_{n\rightarrow\infty}\langle g_{nk}^{(j)}\phi_k^{(j)}, g_{nk}^{(j')}\phi_k^{(j')}\rangle_{\dot{H}^1(\mathbb{R}^2)}=\lim_{n\rightarrow\infty}\langle \Box_kg_{nk}^{(j)}\phi_k^{(j)}, \Box_kg_{nk}^{(j')}\phi_k^{(j')}\rangle_{\dot{H}^1(\mathbb{R}^2)}=0,\\
&\lim_{n\rightarrow\infty}\langle g_{nk}^{(j)}\phi_k^{(j)}, w_{nk}^{(l_k)}\rangle_{\dot{H}^1(\mathbb{R}^2)}=\lim_{n\rightarrow\infty}\langle \Box_kg_{nk}^{(j)}\phi_k^{(j)}, \Box_kw_{nk}^{(l_k)}\rangle_{\dot{H}^1(\mathbb{R}^2)}=0,
\end{aligned}
\end{align*}
 \eqref{N-217} and \eqref{N-218} were proved.

\noindent\textbf{Step 2.}  For arbitrary $(k_1,j_1), (k_2,j_2)\in \Lambda_{1,0}\cup\Lambda_{\infty,1}$, we define
$(k_1,j_1)\sim(k_2,j_2)$ if the orthogonal condition \eqref{N-2} is NOT true for any subsequence, that is
\begin{equation*}
\begin{split}
\frac{\lambda_{nk_1}^{(j_1)}}{\lambda_{nk_2}^{(j_2)}}+\frac{\lambda_{nk_2}^{(j_2)}}{\lambda_{nk_1}^{(j_1)}}
&+\lambda_{nk_2}^{(j_2)}\lambda_{nk_1}^{(j_1)}|\xi_{nk_2}^{(j_2)}-\xi_{nk_1}^{(j_1)}|^2
+\frac{|x_{nk_2}^{(j_2)}-x_{nk_1}^{(j_1)}|^2}{\lambda_{nk_2}^{(j_2)}\lambda_{nk_1}^{(j_1)}}\\
&+\frac{|t_{nk_2}^{(j_2)}(\lambda_{nk_2}^{(j_2)})^2-t_{nk_1}^{(j_1)}(\lambda_{nk_1}^{(j_1)})^2|}{\lambda_{nk_2}^{(j_2)}\lambda_{nk_1}^{(j_1)}}
<\infty\ for\  \forall n.
\end{split}
\end{equation*}
By the definition above, if $(k_1,j_1)\sim(k_2,j_2)$, we have
\begin{equation*}
\lambda_{nk_1}^{(j_1)}\sim \lambda_{nk_2}^{(j_2)},\ \lambda_{nk_1}^{(j_1)}\xi_{nk_1}^{(j_1)} \sim\lambda_{nk_2}^{(j_2)}\xi_{nk_2}^{(j_2)}, \
\frac{x_{nk_1}^{(j_1)}}{\lambda_{nk_1}^{(j_1)}}\sim\frac{x_{nk_2}^{(j_2)}}{\lambda_{nk_2}^{(j_2)}},\
t_{nk_1}^{(j_1)}\sim t_{nk_2}^{(j_2)}.
\end{equation*}
Note that
\begin{equation*}
g_{\theta,\xi_0,x_0,\lambda,t_0}f(x):=\frac{1}{\lambda}e^{i\theta}e^{i\frac{x}{\lambda}
\cdot\lambda\xi_0}(e^{it_0\triangle}f)(\frac{x}{\lambda}-\frac{x_0}{\lambda}),
\end{equation*}
by Remark 4.2, we can put these two profiles together as one profile. Then, we can denote
$(\Lambda_{1,0}\cup\Lambda_{\infty,1})/\sim=\{1,2,\cdots, J\}$,  \eqref{N2-21}-\eqref{N2-3} were proved.

Specially, as $C_c^\infty$ is dense in $L^2$,
we can also suppose $\mathcal {F}\phi_\alpha\in C_c^\infty$ and hence $\phi_\alpha\in H^1(\mathbb{R}^2)$.

 \noindent\textbf{Step 3.} We prove \eqref{N2-4} now.
By  \eqref{N2-3}, we only need to prove that for $\forall m\in\mathbb{N}$, $m\geq2$
\begin{equation}\label{N-15}
\begin{split}
\lim_{J\rightarrow\infty}\limsup_{n\rightarrow\infty}
\{\|u_n\|^{2m}_{L^{2m}(\mathbb{R}^2)}&-\sum_{\alpha=1}^J
\|g_{n\alpha}\phi_\alpha\|^{2m}_{L^{2m}(\mathbb{R}^2)}-
\|R(n,J)\|^{2m}_{L^{2m}(\mathbb{R}^2)}\}=0.
\end{split}
\end{equation}
As
\begin{equation*}
\|R(n,J)\|_{L^{2m}(\mathbb{R}^2)}\lesssim \|R(n,J)\|^{1/(m-1)}_{L^{4}(\mathbb{R}^2)}\|R(n,J)\|^{(m-2)/(m-1)}_{H^{1}(\mathbb{R}^2)}
\end{equation*}
and for $\frac{1}{4}<\theta<\frac{1}{2}$,
\begin{equation*}
\begin{split}
\|R(n,J)\|_{L^{4}(\mathbb{R}^2)}&\leq \|e^{it\triangle}R(n,J)\|_{L_t^\infty L_x^{4}(\mathbb{R}\times\mathbb{R}^2)}
\lesssim \|\langle\partial_t\rangle^\theta e^{it\triangle}R(n,J)\|_{L_{t,x}^{4}(\mathbb{R}\times\mathbb{R}^2)}\\
&=\|e^{it\triangle}R(n,J)\|^{1-2\theta}_{L_{t,x}^4(\mathbb{R}\times\mathbb{R}^2)}\|\langle\partial_t\rangle^{1/2} e^{it\triangle}R(n,J)\|^{2\theta}_{ L_{t,x}^{4}(\mathbb{R}\times\mathbb{R}^2)}\\
&=\|e^{it\triangle}R(n,J)\|^{1-2\theta}_{L_{t,x}^4(\mathbb{R}\times\mathbb{R}^2)}\|e^{it\triangle}R(n,J)\|^{2\theta}_{ L_{t}^{4}H_4^1(\mathbb{R}\times\mathbb{R}^2)}\\
&\lesssim\|e^{it\triangle}R(n,J)\|^{1-2\theta}_{L_{t,x}^4(\mathbb{R}\times\mathbb{R}^2)},
\end{split}
\end{equation*}
by \eqref{N2-1}, we have
\begin{equation*}
\lim_{J\rightarrow\infty}\limsup_{n\rightarrow\infty}\|R(n,J)\|^{2m}_{L^{2m}(\mathbb{R}^2)}=0
\end{equation*}

We separate the set $1\leq\alpha\leq J$ into two subsets:
\begin{equation}\label{N2-7}
\Lambda_{1}=\{1\leq\alpha\leq J\ |\ \lambda_{n\alpha}\equiv1, \xi _{n\alpha}\equiv0\},\ \
\Lambda_{\infty}=\{1\leq\alpha\leq J\ |\ \lambda_{n\alpha}\rightarrow\infty\}.
\end{equation}

When $\alpha\in\Lambda_{\infty}$,
\begin{equation}
\begin{split}
\lim_{n\rightarrow\infty}\|g_{n\alpha}\phi_\alpha\|_{L^{2m}(\mathbb{R}^2)}
&=\lim_{n\rightarrow\infty}(\lambda_{n\alpha})^{-1+1/m}\|e^{it_{n\alpha}\triangle}\phi_\alpha\|_{L^{2m}(\mathbb{R}^2)}\\
&\leq\lim_{n\rightarrow\infty}(\lambda_{n\alpha})^{-1+1/m}\|\phi_\alpha\|_{H^{1}(\mathbb{R}^2)}=0.
\end{split}
\end{equation}

Hence, in order  to prove \eqref{N-15}, one only need to prove
\begin{equation}\label{N-116}
\lim_{n\rightarrow\infty}\{\|\sum_{\alpha\in\Lambda_{1}}g_{n\alpha}\phi_\alpha\|^{2m}_{L^{2m}(\mathbb{R}^2)}
-\sum_{\alpha\in\Lambda_{1}}
\|g_{n\alpha}\phi_\alpha\|^{2m}_{L^{2m}(\mathbb{R}^2)}\}=0.
\end{equation}

If $\alpha\in\Lambda_{1}$ and $t_{n\alpha}\rightarrow\infty$,
for a function $\tilde{\phi}_\alpha\in\dot{H}^{1/2}\cap L^{4/3}$ we have
\begin{equation}
\begin{split}
\|g_{n\alpha}\phi_\alpha\|_{L^{4}(\mathbb{R}^2)}&\leq \|g_{n\alpha}\phi_\alpha-g_{n\alpha}\tilde{\phi}_\alpha\|_{L^{4}(\mathbb{R}^2)}+\|g_{n\alpha}\tilde{\phi}_\alpha\|_{L^{4}(\mathbb{R}^2)}\\
&\leq \|\phi_\alpha-\tilde{\phi}_\alpha\|_{\dot{H}^{1/2}(\mathbb{R}^2)}+|t_{n\alpha}|^{-1/2}\|\tilde{\phi}_\alpha\|_{L^{4/3}(\mathbb{R}^2)}
\end{split}
\end{equation}
By approximating $\phi_\alpha$ by $\tilde{\phi}_\alpha\in C_c^\infty$ in $\dot{H}^{1/2}$ and sending $n\rightarrow\infty$ we have
$\|g_{n\alpha}\phi_\alpha\|_{L^{4}(\mathbb{R}^2)}\rightarrow0$. Note that $g_{n\alpha}\phi_\alpha\in H^1$, we obtain $\|g_{n\alpha}\phi_\alpha\|_{L^{2m}(\mathbb{R}^2)}\rightarrow0$ for $\forall m\geq2$.

If $\alpha\in\Lambda_{1}$ and $t_{n\alpha}\equiv 0$,
we have orthogonal condition $|x_{n\alpha}-x_{n\alpha'}|\rightarrow\infty$ for any $\alpha\neq\alpha'$. Thus
\begin{equation*}
\begin{split}
&\lim_{n\rightarrow\infty}\{\|\sum_{\alpha\in\Lambda_{1},t_{n\alpha}\equiv 0}g_{n\alpha}\phi_\alpha\|^{2m}_{L^{2m}(\mathbb{R}^2)}
-\sum_{\alpha\in\Lambda_{1},t_{n\alpha}\equiv 0}
\|g_{n\alpha}\phi_\alpha\|^{2m}_{L^{2m}(\mathbb{R}^2)}\}\\
=&\lim_{n\rightarrow\infty}\{\|\sum_{\alpha\in\Lambda_{1},t_{n\alpha}\equiv 0}\phi_\alpha(\cdot-x_{n\alpha})\|^{2m}_{L^{2m}(\mathbb{R}^2)}-\sum_{\alpha\in\Lambda_{1},t_{n\alpha}\equiv 0}
\|\phi_\alpha(\cdot-x_{n\alpha})\|^{2m}_{L^{2m}(\mathbb{R}^2)}\}=0.
\end{split}
\end{equation*}
\eqref{N-116} holds and then \eqref{N2-4} was proved.

\end{proof}

\section{The Proof  of  Theorem \ref{th1.2}}

Let $u$ be a solution of \eqref{1.1b}, by Strichartz estimate and \eqref{2.038},
\begin{equation}
\begin{split}
&\|u\|_{L^4_{t,x}\cap L_t^{2/(1-2\varepsilon)}L_x^{1/\varepsilon}(\mathbb{R}\times\mathbb{R}^2)}\\
\leq &C\|u_0\|_{L^2(\mathbb{R}^2)}
+C\|u\|^3_{L^4_{t,x}(\mathbb{R}\times\mathbb{R}^2)}
+C(E)\|u\|_{L_t^{2/(1-2\varepsilon)}L_x^{1/\varepsilon}(\mathbb{R}\times\mathbb{R}^2)}\|u\|^{8\varepsilon}_{L^4_{t,x}(\mathbb{R}\times\mathbb{R}^2)}\\
\leq &C\|u_0\|_{L^2(\mathbb{R}^2)}+C\|u\|^3_{L^4_{t,x}\cap L_t^{2/(1-2\varepsilon)}L_x^{1/\varepsilon}(\mathbb{R}\times\mathbb{R}^2)}+C(E)\|u\|^{1+8\varepsilon}_{L^4_{t,x}\cap L_t^{2/(1-2\varepsilon)}L_x^{1/\varepsilon}(\mathbb{R}\times\mathbb{R}^2)}.
\end{split}
\end{equation}
When $\|u_0\|_{L^2(\mathbb{R}^2)}\ll 1$, by standard continuity argument, we have
\begin{equation}\label{18.2}
\|u\|_{L_{t,x}^4(\mathbb{R}\times\mathbb{R}^2)}\leq C\|u_0\|_{L^2(\mathbb{R}^2)}<\infty.
\end{equation}
Hence, if $M(u)\ll 1$, then $\|u\|_{L_{t,x}^4(\mathbb{R}\times\mathbb{R}^2)}<\infty$. In particular,
we have global existence and scattering in both directions.

For any mass $m\geq0$, we define
\begin{equation*}
A(m):= \sup \{S(u):u\ is\ the\ global\ sulution\ of\  \eqref{1.1b}, M(u)\leq m, H(u)\leq1\}.
\end{equation*}
Then $A:[0,+\infty)\rightarrow [0,+\infty]$ is a monotone increasing
function of $m$. As $A$  is left-continuous  and finite for small $m$, there must exist a unique critical mass $m_0\in(0 +\infty]$
such that $A(m)$ is finite for all $m<m_0$ but infinite for all $m\geq m_0$.



To prove Theorem \ref{th1.2}, one only needs to prove that the critical mass $m_0$ is infinite. We will  prove that by contradiction.
\begin{Proposition}\label{Prop1}
Suppose that the critical mass $m_0$ is finite. Let $u_n:\mathbb{R}\times\mathbb{R}^2\rightarrow \mathbb{C}$ for $n=1,2,\cdots$ be a sequence of solutions and $t_n\in \mathbb{R}$ be a sequence of times such that $\limsup_{n\rightarrow \infty}M(u_n)=m_0$ and
\begin{equation}\label{B-1.3}
\lim_{n\rightarrow\infty}S_{\geq t_n}(u_n)=\lim_{n\rightarrow\infty}S_{\leq t_n}(u_n)=+\infty.
\end{equation}
Then there exists a sequence of $x_n=x_n(t_n)\in \mathbb{R}^2$such that $u_n(t_n, x+x_n)$ has a subsequence which converges strongly in $L^2_x(\mathbb{R}^2)$.
\end{Proposition}

\begin{proof}


We can take $t_n=0$ for all $n$ by translating $u_n$ in time. Thus,
\begin{equation}\label{B-1.12}
\lim_{n\rightarrow\infty}S_{\geq 0}(u_n)=\lim_{n\rightarrow\infty}S_{\leq 0}(u_n)=+\infty.
\end{equation}
By Lemma \ref{N2-Lem1}, up to a subsequence if necessary, we have
\begin{equation*}
u_n(0)=\sum_{\alpha\in\Lambda_{1}\cup\Lambda_{\infty}}g_{n\alpha}\phi_\alpha+R(n,J),
\end{equation*}
where $\Lambda_{1} $ and $\Lambda_{\infty}$ were defined by \eqref{N2-7}.
Suppose that
\begin{equation*}
g_{n\alpha}=h_{n\alpha}e^{it_{n\alpha}\triangle}
\end{equation*}
where $t_{n\alpha}\in\mathbb{R}$ and $h_{n\alpha}\in G$.
By \eqref{N2-2},
\begin{equation}\label{B-1.14}
\sum_{\alpha\in\Lambda_{1}\cup\Lambda_{\infty}}
M(\phi_\alpha)\leq \limsup_{n\rightarrow\infty}M(u_n(0))\leq m_0
\end{equation}
Hence,
\begin{equation}\label{N2-50}
\sup_{\alpha\in\Lambda_{1}\cup\Lambda_{\infty}}M(\phi_\alpha)\leq m_0.
\end{equation}

We define the nonlinear profile $v_\alpha:\mathbb{R}\times\mathbb{R}^2\rightarrow \mathbb{C}$ as follows:

$\blacklozenge$ When $\alpha\in\Lambda_{1}$

$\diamond$ if $t_{n\alpha}\equiv0$, we define $v_\alpha$ to be the global solution of \eqref{1.1b} with initial data $v_\alpha(0)=\phi_\alpha$.

$\diamond$ if $t_{n\alpha}\rightarrow+\infty$, we define $v_\alpha$ to be the global  solution of \eqref{1.1b} which scatters to $e^{it\triangle}\phi_\alpha$ when $t\rightarrow+\infty$.

$\diamond$ if $t_{n\alpha}\rightarrow-\infty$, we define $v_\alpha$ to be the  global  solution of \eqref{1.1b} which scatters to $e^{it\triangle}\phi_\alpha$ when $t\rightarrow-\infty$.

$\blacklozenge$ When $\alpha\in\Lambda_{\infty}$,

$\diamond$ if $t_{n\alpha}\equiv0$ , we define $v_\alpha$ to be the  global  solution of $iu_t+\triangle u=|u|^2u$ with initial data $v_\alpha(0)=\phi_\alpha$.

$\diamond$ if $t_{n\alpha}\rightarrow+\infty$, we define $v_\alpha$ to be the  global  solution of $iu_t+\triangle u=|u|^2u$ which scatters to $e^{it\triangle}\phi_\alpha$ when $t\rightarrow+\infty$.

$\diamond$ if $t_{n\alpha}\rightarrow-\infty$, we define $v_\alpha$ to be the  global  solution of $iu_t+\triangle u=|u|^2u$ which scatters to $e^{it\triangle}\phi_\alpha$ when $t\rightarrow-\infty$.

If we define
\begin{equation}\label{B-1.19}
\tilde{u}_n=\sum_{\alpha\in\Lambda_{1}\cup\Lambda_{\infty}}T_{h_{n\alpha}}[v_\alpha(\cdot+t_{n\alpha})](t)+e^{it\triangle}R(n,J)
\end{equation}
for $n,J=1,2,\cdots$,
then we have the following two lemmas:

\begin{Lemma} \label{N2-Lem2}(\cite{B-T-M-Z}, Lemma 5.1)
\begin{equation*}
\lim_{n\rightarrow\infty}M(\tilde{u}_n(0)-u_n(0))=0.
\end{equation*}
\end{Lemma}
\begin{Lemma} \label{N2-Lem3}
If
\begin{equation*}
\lim_{J\rightarrow\infty}\limsup_{n\rightarrow\infty}\|\tilde{u}_n\|_{L_{t,x}^{4}(\mathbb{R}\times\mathbb{R}^2)}<\infty,\ \ \|v_\alpha\|_{L_{t,x}^{4}(\mathbb{R}\times\mathbb{R}^2)}<\infty\  ( \forall\alpha),
\end{equation*}
 then
\begin{equation*}
\lim_{J\rightarrow\infty}\limsup_{n\rightarrow\infty}\|\int_0^te^{i(t-\tau)\triangle}
((i\partial_t+\triangle)\tilde{u}_n-f(\tilde{u}_n))(\tau)d\tau\|_X=0.
\end{equation*}
where $X=L_{t,x}^4\cap L_t^{2/(1-2\varepsilon)}L_x^{1/\varepsilon}(\mathbb{R}\times\mathbb{R}^2)$.
\end{Lemma}
\begin{proof}
Denote
\begin{equation*}
v_{n\alpha}=T_{h_{n\alpha}}[v_\alpha(\cdot+t_{n\alpha})].
\end{equation*}
By the definition of $\tilde{u}_n$, we have
\begin{equation*}
\tilde{u}_n=\sum_{\alpha\in\Lambda_{1}\cup\Lambda_{\infty}}v_{n\alpha}+e^{it\triangle}R(n,J)
\end{equation*}
and
\begin{equation*}
(i\partial_t+\triangle)\tilde{u}_n=\sum_{\alpha\in\Lambda_{1}}f(v_{n\alpha}) +\sum_{\alpha\in\Lambda_{\infty}}|v_{n\alpha}|^2v_{n\alpha}.
\end{equation*}
Thus, by triangle inequality, it suffices to show that
\begin{equation}\label{B-1.21}
\lim_{J\rightarrow\infty}\limsup_{n\rightarrow\infty}\|\int_0^te^{i(t-\tau)\triangle}
(f(\tilde{u}_n-e^{it\triangle}R(n,J))-f(\tilde{u}_n))(\tau)d\tau\|_X=0,
\end{equation}
\begin{equation}\label{B-1.22}
\lim_{n\rightarrow\infty}\|\int_0^te^{i(t-\tau)\triangle}
\left(f(\sum_{\alpha\in\Lambda_{1}\cup\Lambda_{\infty}}v_{n\alpha})-\sum_{\alpha\in\Lambda_{1}\cup\Lambda_{\infty}}f(v_{n\alpha})
\right)(\tau)d\tau\|_X=0
\end{equation}
and
\begin{equation}\label{N-13}
\lim_{n\rightarrow\infty}\|\int_0^te^{i(t-\tau)\triangle}
\sum_{\alpha\in\Lambda_{\infty}}\left(f(v_{n\alpha})-|v_{n\alpha}|^2v_{n\alpha}\right)(\tau)d\tau\|_X=0
\end{equation}

Using \eqref{N2-4} and the same estimates as in \eqref{B-1.6}, we have
\begin{equation*}
\begin{split}
&\|\int_0^te^{i(t-\tau)\triangle}
(f(\tilde{u}_n-e^{it\triangle}R(n,J))-f(\tilde{u}_n))(\tau)d\tau\|_X\\
\lesssim&(\|\tilde{u}_n\|^2_{L_{t,x}^4(\mathbb{R}\times\mathbb{R}^2)}+\|e^{it\triangle}R(n,J)\|^2_{L_{t,x}^4(\mathbb{R}\times\mathbb{R}^2)})
\|e^{it\triangle}R(n,J)\|_{L_{t,x}^4(\mathbb{R}\times\mathbb{R}^2)}\\
+&(\|\tilde{u}_n\|^{8\varepsilon}_{L_{t,x}^4(\mathbb{R}\times\mathbb{R}^2)}+
\|e^{it\triangle}R(n,J)\|^{8\varepsilon}_{L_{t,x}^4(\mathbb{R}\times\mathbb{R}^2)})
\|e^{it\triangle}R(n,J)\|_{L_t^{2/(1-2\varepsilon)}L_{x}^{1/\varepsilon}(\mathbb{R}\times\mathbb{R}^2)}.
\end{split}
\end{equation*}
By \eqref{N2-1},
\begin{equation*}
\lim_{J\rightarrow\infty}\limsup_{n\rightarrow\infty}\|e^{it\triangle}R(n,J)\|_{L_{t,x}^4(\mathbb{R}\times\mathbb{R}^2)}=0.
\end{equation*}
As
\begin{equation*}
\begin{split}
&\|e^{it\triangle}R(n,J)\|_{L_t^{2/(1-2\varepsilon)}L_{x}^{1/\varepsilon}(\mathbb{R}\times\mathbb{R}^2)}\\
\lesssim&\|e^{it\triangle}R(n,J)\|^{2\varepsilon/(1-2\varepsilon)}_{L_{t,x}^4(\mathbb{R}\times\mathbb{R}^2)}
\|e^{it\triangle}R(n,J)\|^{(1-4\varepsilon)/(1-2\varepsilon)}_{L_t^{2/(1-\varepsilon)}L_x^{2/\varepsilon}(\mathbb{R}\times\mathbb{R}^2)},
\end{split}
\end{equation*}
and
\begin{equation*}
\|e^{it\triangle}R(n,J)\|_{L_t^{2/(1-\varepsilon)}L_x^{2/\varepsilon}(\mathbb{R}\times\mathbb{R}^2)}
\lesssim \|R(n,J)\|_{L_x^2(\mathbb{R}^2)}<\infty,
\end{equation*}
we have
\begin{equation*}
\lim_{J\rightarrow\infty}\limsup_{n\rightarrow\infty}
\|e^{it\triangle}R(n,J)\|_{L_t^{2/(1-2\varepsilon)}L_{x}^{1/\varepsilon}(\mathbb{R}\times\mathbb{R}^2)}=0.
\end{equation*}
By \eqref{B-1.20}, \eqref{B-1.21} was obtained.

Using Strichartz estimate,
\begin{equation*}
\begin{split}
&\|\int_0^te^{i(t-\tau)\triangle}
\left(f(v_{n\alpha})-|v_{n\alpha}|^2v_{n\alpha}\right)(\tau)d\tau\|_X\\
\leq&\|\sum_{m=2}^\infty\frac{|v_{n\alpha}|^{2m}v_{n\alpha}}{m!}\|_{L_{t,x}^{4/3}(\mathbb{R}\times\mathbb{R}^2)}
\leq\sum_{m=2}^\infty\frac{1}{m!}\|v_{n\alpha}\|_{L_{t,x}^{4}(\mathbb{R}\times\mathbb{R}^2)}
\|v_{n\alpha}\|^{2m}_{L_{t,x}^{4m}(\mathbb{R}\times\mathbb{R}^2)}\\
\leq &\sum_{m=2}^\infty\frac{1}{m!}(\lambda_{n\alpha})^{-2(m-1)}\|v_\alpha\|_{L_{t,x}^{4}(\mathbb{R}\times\mathbb{R}^2)}
\|v_\alpha\|^{2m}_{L_{t,x}^{4m}(\mathbb{R}\times\mathbb{R}^2)}.
\end{split}
\end{equation*}
By Lemma \ref{B-Lem2.3},
\begin{equation*}
\|v_\alpha\|_{L_{t,x}^{4m}(\mathbb{R}\times\mathbb{R}^2)}\leq Cm^{1/2+1/2m}\|v_\alpha\|^{1/m}_{L_{t,x}^{4}(\mathbb{R}\times\mathbb{R}^2)}
\|v_\alpha\|^{1-1/m}_{L^{\infty}(\mathbb{R},\dot{H}^1)}.
\end{equation*}
Note that $\|v_\alpha\|_{L_{t,x}^{4}(\mathbb{R}\times\mathbb{R}^2)}<\infty$, \eqref{N-13} was obtained.

To prove Lemma \ref{N2-Lem3}, we only left to prove \eqref{B-1.22}.
Note that
\begin{equation*}
\begin{split}
\left|f(\sum_{\alpha=1}^Jz_\alpha)-\sum_{\alpha=1}^Jf(z_\alpha)\right|
\lesssim\sum_{\alpha\neq \alpha'}|z_\alpha||e^{\lambda|z_{\alpha'}|^2}-1|,
\end{split}
\end{equation*}
by \eqref{N2-4}, we have
\begin{equation*}
\begin{split}
&\|\int_0^te^{i(t-\tau)\triangle}
\left(f(\sum_{\alpha\in\Lambda_{1}\cup\Lambda_{\infty}}v_{n\alpha})-\sum_{\alpha\in\Lambda_{1}\cup\Lambda_{\infty}}f(v_{n\alpha})
\right)(\tau)d\tau\|_X\\
\lesssim &\sum_{\alpha\neq \alpha'}\left(\|v_{n\alpha}|v_{n\alpha'}|^2\|_{L_{t,x}^{4/3}(\mathbb{R}\times\mathbb{R}^2)}\right. \\
&\left.+\|v_{n\alpha}(e^{\lambda|v_{n\alpha'}|^2}-\lambda|v_{n\alpha'}|^2-1)\|_
{L_t^{2/(1+2\varepsilon)}L_x^{1/(1-\varepsilon)}(\mathbb{R}\times\mathbb{R}^2)}\right)\\
\lesssim &\sum_{\alpha\neq \alpha'}\||v_{n\alpha}|^{1/2}|v_{n\alpha'}|^{1/2}\|^2_{L_{t,x}^4(\mathbb{R}\times\mathbb{R}^2)}
\|v_{n\alpha'}\|_{L_{t,x}^4(\mathbb{R}\times\mathbb{R}^2)}\\
&+\sum_{\alpha\neq \alpha'}\||v_{n\alpha}|^{\varepsilon/(1+\varepsilon)}|v_{n\alpha'}|^{1/(1+\varepsilon)}\|
^{\frac{4\varepsilon(1+\varepsilon)}{(1-2\varepsilon)}}_{L_{t,x}^4(\mathbb{R}\times\mathbb{R}^2)} \|v_{n\alpha}\|^{\frac{1-2\varepsilon-4\varepsilon^2}{1-2\varepsilon}}_{L_t^{\frac{2(1-2\varepsilon-4\varepsilon^2)}{1-2\varepsilon-6\varepsilon^2}}
L_x^{\frac{1-2\varepsilon-4\varepsilon^2}{\varepsilon^2}}(\mathbb{R}\times\mathbb{R}^2)}\\
&\cdot\||v_{n\alpha'}|^{-4\varepsilon/(1-2\varepsilon)}(e^{\lambda|v_{n\alpha'}|^2}-\lambda|v_{nk'}^{(j')}|^2-1)\|_
{L_t^\infty L_x^{(1-2\varepsilon)/(1-4\varepsilon)}(\mathbb{R}\times\mathbb{R}^2)}.
\end{split}
\end{equation*}
By \eqref{B-1.10}, we immediately obtain \eqref{B-1.22}.
\end{proof}

By \eqref{N2-50},
suppose
\begin{equation}\label{B-1.15}
\sup_{\alpha\in\Lambda_{1}\cup\Lambda_{\infty}}M(\phi_\alpha)\leq m_0-\sigma
\end{equation}
for some $\sigma>0$, we will prove that this leads to a contradiction.
By the definition of $A(m)$ and \eqref{18.2}, we have
\begin{equation}\label{B-1.16}
A(m)\leq Bm \ for\  all\  0\leq m\leq m_0-\sigma,
\end{equation}
where $B=B(\sigma)\in(0,+\infty)$.
Then $v_\alpha$ satisfies
\begin{equation}\label{B-1.17}
M(v_\alpha)=M(\phi_\alpha)\leq m_0-\sigma,
\end{equation}
\begin{equation}\label{B-1.18}
S(v_\alpha)\leq A(M(\phi_\alpha))\leq BM(\phi_\alpha).
\end{equation}
By \eqref{B-1.11}, \eqref{B-1.14} and \eqref{B-1.18}, we have
\begin{equation}\label{B-1.20}
\lim_{J\rightarrow\infty}\limsup_{n\rightarrow\infty}S(\tilde{u}_n)
\leq Bm_0.
\end{equation}

Using Lemma \ref{N2-Lem2} and Lemma \ref{N2-Lem3}, we have
\begin{equation*}
\begin{split}
&M(\tilde{u}_n(0)-u_n(0))\leq\delta,\ \ \ \ S(\tilde{u}_n)\leq 2Bm_0,\\
&\|\int_0^te^{i(t-\tau)\triangle}
((i\partial_t+\triangle)\tilde{u}_n-f(\tilde{u}_n))(\tau)d\tau\|_{L_{t,x}^4\cap L_t^{2/(1-2\varepsilon)}L_x^{1/\varepsilon}(\mathbb{R}\times\mathbb{R}^2)}\leq\delta,
\end{split}
\end{equation*}
for $\delta>0$ sufficiently small, $J=J(\delta)$ and $n=n(J,\delta)$ sufficiently large.
By Lemma \ref{B-Lem1.4}, we obtain that
$S(u_n)\leq3Bm_0$ which contradicts \eqref{B-1.12}.
Thus, \eqref{B-1.15} fails for $\forall\sigma>0$, and then
\begin{equation*}
\sup_{\alpha\in\Lambda_{1}\cup\Lambda_{\infty}}M(\phi_\alpha)=m_0.
\end{equation*}
Comparing this with \eqref{B-1.14}, we have
\begin{equation}\label{B-1.23}
u_n(0)=h_ne^{it_n\triangle}\phi+R_n
\end{equation}
with $t_n$ converging to $\pm\infty$ or $t_n\equiv0$, $h_n\in G$, $M(\phi)=m_0$ and
$$
M(R_n)\rightarrow 0,\  S(e^{it\triangle}R_n)\rightarrow0 \ \ as \ \ n\rightarrow\infty.
$$
Specially, the parameters $\lambda_n, \xi_n$ of $h_n$ must satisfies
\begin{align*}
\begin{aligned}
 \lambda_{n}\equiv1 \ and\  \xi _{n}\equiv0, \ or \ \lambda_{n}\rightarrow\infty.
\end{aligned}
\end{align*}

Since there is only one profile now,
 we have
\begin{equation*}
\tilde{u}_n=T_{h_{n}}v+e^{it\triangle}R_n.
\end{equation*}
When $\lambda_n\rightarrow\infty$,
by the scattering of cubic Schr\"{o}dinger equation (see \cite{B-K-T-V}, \cite{Dodson}), we have $S(v)<\infty$ and $\lim_{n\rightarrow\infty}S(\tilde{u}_n)<\infty$.
By Lemma \ref{N2-Lem2}, Lemma \ref{N2-Lem3} and Lemma \ref{B-Lem1.4}, we obtain that for $n$ sufficiently large,
$S(u_n)<\infty$ which contradicts \eqref{B-1.12}.

When $\lambda_n\equiv1$, $\xi_n\equiv0$ and $t_n\rightarrow+\infty$,
by Strichartz estimate and monotone convergence we have
\begin{equation*}
\lim_{n\rightarrow\infty}S_{\geq0}(e^{it\triangle}e^{it_n\triangle}\phi)=0.
\end{equation*}
Thus
\begin{equation*}
\lim_{n\rightarrow\infty}S_{\geq0}(e^{it\triangle}h_ne^{it_n\triangle}\phi)
=\lim_{n\rightarrow\infty}S_{\geq0}(e^{it\triangle}e^{it_n\triangle}\phi)=0.
\end{equation*}
Since $\lim_{n\rightarrow\infty} S(e^{it\triangle}R_n)=0$,  we can see from \eqref{B-1.23} that
\begin{equation*}
\lim_{n\rightarrow\infty}S_{\geq0}(e^{it\triangle}u_n(0))=0.
\end{equation*}
By Lemma \ref{B-Lem1.4} (with $0$ as the approximate solution and $u_n(0)$ as the initial data),
we have
\begin{equation*}
\lim_{n\rightarrow\infty}S_{\geq0}(u_n)=0
\end{equation*}
which contradicts one of the estimates in \eqref{B-1.12}.

When $\lambda_n\equiv1$,  $\xi_n\equiv0$ and $t_n\rightarrow-\infty$, the argument is similar and we can obtain a contradiction by using the other half of \eqref{B-1.12}.

Now, the only case left is $\lambda_n\equiv1$,  $\xi_n\equiv0$ and $t_n\equiv0$, .
In this case, we have
\begin{equation*}
M(u_n(0)-h_n\phi)=M(R_n)\rightarrow 0 \ \ as \ \ n\rightarrow\infty.
\end{equation*}
Thus $(h_n)^{-1}u_n(0)=e^{i\theta_n}u_n(0, x+x_n)$ converges to $\phi$ in $L^2_x(\mathbb{R}^2)$.
After passing to a subsequence if necessary and refining $\phi$,  the desired result follows.
\end{proof}

Let $\{u_n\}$ be the sequence given in Proposition \ref{Prop1} and suppose $u_n(0, x+x_n)$ converges to $u_0$ strongly in $L^2_x(\mathbb{R}^2)$, then $M(u_0)\leq m_0$. Let $u$  be the global
solution with initial data $u(0)=u_0$, by Lemma \ref{B-Lem1.4}, we must have
\begin{equation*}
S_{\geq 0}(u)=S_{\leq 0}(u)=+\infty.
\end{equation*}
By the definition of $m_0$, $M(u_0)\geq m_0$ and hence $M(u_0)= m_0$.

 Since $u$ is locally in $L_{t,x}^4$, for $\forall t_n\in \mathbb{R}$, we have
\begin{equation*}
S_{\geq t_n}(u)=S_{\leq t_n}(u)=+\infty.
\end{equation*}
Using Proposition \ref{Prop1} for $\{u(t_n)\}$, we have $u(t_n, x+x(t_n))$ converges in $L^2_x(\mathbb{R}^2)$. By Ascoli-Arzela Theorem,
that is
\begin{Proposition}\label{Prop2}
Suppose that the critical mass $m_0$ is finite. Then there exists a global solution $u$ of mass exactly $m_0$ satisfies that for every $\eta>0$ there exists $0<C(\eta)<\infty$ such that
\begin{equation}\label{B-1.1}
\int_{|x-x(t)|\geq C(\eta)}|u(t,x)|^2dx+\int_{|\xi-\xi(t)|\geq C(\eta)}|\hat{u}(t,\xi)|^2d\xi\leq \eta
\end{equation}
for all $t\in \mathbb{R}$,
where the  functions $x,\xi:\mathbb{R}\rightarrow \mathbb{R}^2$.

\end{Proposition}

\begin{Proposition}\label{Prop3}
The solution described in Proposition \ref{Prop2} does not exist.
\end{Proposition}

Once we proved Proposition \ref{Prop3}, we can say that $m_0=\infty$
and thus Theorem \ref{th1.2} is true.
In order to prove Proposition \ref{Prop3}, we need the following two lemmas.
\begin{Lemma}\label{Lem4.1}{\rm(\cite{Naka99}, Lemma 5.2)}
Let u be a global solution of \eqref{1.1b}. Then we have
\begin{equation}\label{4.1}
\iint_{\mathbb{R}\times\mathbb{R}^2}\frac{\langle t\rangle^2G(u)}{\langle t\rangle^3+|x|^3}dxdt\leq C(E).
\end{equation}
\end{Lemma}

%
%
%
\begin{Lemma}\label{Lem5.2}{\rm(\cite{Naka99}, Lemma 6.2)}
Let u be a global solution of \eqref{1.1b}. Let $B$ be a compact subset of $\mathbb{R}^2$. Then for any $R>0$
and $T>0$, we have
\begin{equation}\label{5.2}
\int_{B(R)}|u(T,x)|^2dx\geq\int_B|u(0,x)|^2dx-C(E)T/R,
\end{equation}
where $B(R):=\{x\in\mathbb{R}^2|\exists y\in B \ s.t.\  |x-y|<R\}$.
\end{Lemma}

$The\  proof\  of \ Proposition\  \ref{Prop3}.$
By Lemma \ref{Lem5.2}, choosing $\eta$ sufficiently small,
\begin{equation*}
\begin{split}
\int_{|x-x(0)|\leq C(\eta)+R|t|}|u(t,x)|^2dx&\geq\int_{|x-x(0)|\leq C(\eta)}|u(0,x)|^2dx-C(E)/R\\
&\geq m_0-\eta-C(E)/R.
\end{split}
\end{equation*}
By Proposition \ref{Prop2},
\begin{equation*}
\int_{|x-x(t)|\leq C(\eta)}|u(t,x)|^2dx\geq m_0-\eta.
\end{equation*}
For a fixed large number $R$, we must have $|x(t)-x(0)|\leq 2C(\eta)+R|t|$.
By Lemma \ref{Lem4.1} and H\"{o}lder inequality,
\begin{equation*}
\begin{split}
\infty&>\iint_{\mathbb{R}\times\mathbb{R}^2}\frac{\langle t\rangle^2|u|^4}{\langle t\rangle^3+|x|^3}dxdt\\
&\gtrsim \int_{\mathbb{R}}\int_{|x-x(t)|\leq C(\eta)}\frac{\langle t\rangle^2|u|^4}{\langle t\rangle^3+|x|^3}dxdt\\
&\gtrsim\int_{\mathbb{R}}\int_{|x-x(t)|\leq C(\eta)}\frac{|u|^4}{\langle t\rangle+1}dxdt\\
&\gtrsim\int_{\mathbb{R}}\int_{|x-x(t)|\leq C(\eta)}\frac{|u|^2}{\langle t\rangle+1}dxdt\\
&\gtrsim (m_0-\eta)\int_{\mathbb{R}}\frac{1}{\langle t\rangle+1}dxdt=\infty.
\end{split}
\end{equation*}
This is a contradiction. Proposition \ref{Prop3} was obtained.
\hfill $\square$

\end{document}